\documentclass[11pt,a4paper]{article}

\usepackage[a4paper, total={5.7in, 9in}]{geometry}

\usepackage{amsmath,amsfonts}
\usepackage{amsthm,amssymb,hyperref}
\usepackage[english]{babel}
\usepackage{enumitem}

\usepackage{xcolor} 
\hypersetup{
    colorlinks,
    linkcolor={blue},
    citecolor={red},
    urlcolor={blue}
}

\newtheorem{theorem}{Theorem}
\newtheorem{lemma}[theorem]{Lemma}
\newtheorem{conjecture}[theorem]{Conjecture}
\newtheorem{corollary}[theorem]{Corollary}
\newtheorem{proposition}[theorem]{Proposition}

\newtheorem{claim}[theorem]{Claim}
\numberwithin{theorem}{section}

\newcommand{\PP}{\mathcal{P}}
\newcommand{\EE}{\mathcal{E}}
\newcommand{\FF}{\overline{F}}

\begin{document}

\title{Locally self-avoiding eulerian tours}

\author{Tien-Nam Le\\~\\
			\small Laboratoire d'Informatique du Parall\'elisme \\ \small \'Ecole Normale Sup\'erieure de Lyon \\ \small 69364 Lyon Cedex 07, France}

\date{}

\maketitle

\begin{abstract}
It was independently conjectured by H\"aggkvist in 1989 and Kriesell in 2011 that given a positive integer $\ell$, every simple eulerian graph with high minimum degree (depending on $\ell$) admits an eulerian tour such that every 
segment of length at most $\ell$ of the tour is a path. 
Bensmail, Harutyunyan, Le and Thomass\'e recently verified the conjecture for 4-edge-connected eulerian graphs. Building on that proof, we prove here the full statement of the conjecture.
This implies a variant of the path case of Bar\'at-Thomassen conjecture that any simple eulerian graph with high minimum degree can be decomposed into paths of fixed length and possibly an additional shorter path.
\end{abstract}

\section{Introduction} \label{section:introduction}

Unless stated otherwise, graphs considered here are simple and undirected, while multigraphs may contain multiple edges and loops, where each loop contributes two to the degree of the incident vertex. 
Given an eulerian tour $\EE$ of a multigraph $G$, for every positive integer $\ell$, a walk $e_1e_2...e_\ell$ where any $e_i,e_{i+1}$ are consecutive edges of $\EE$ is called a \emph{segment} of \emph{length} $\ell$ of $\EE$. We say that $\EE$ is \emph{$\ell$-step self-avoiding} if every segment of length at most $\ell$ of $\EE$ is a path, which is equivalent to that $\EE$ ``contains" no cycle of length at most $\ell$. 

H\"aggkvist (\cite{H89}, Problem 3.3) and Kriesell \cite{Kr} independently conjectured that high minimum degree is a sufficient condition for the existence of an $\ell$-step self-avoiding eulerian tour. 

\begin{conjecture}[\cite{H89,Kr}] \label{conj:eulerian}
For every positive integer $\ell$, there is an integer $d_\ell$ such that every eulerian graph
$G$ with minimum degree at least~$d_\ell$ admits an $\ell$-step self-avoiding eulerian tour.
\end{conjecture}

H\"aggkvist also asked to identify the minimum of $d_\ell$  if it exists. For the case $\ell=3$, i.e. triangle-free eulerian tours, Adelgren \cite{A95} characterized all graphs with maximum degree at most 4 which admit a triangle-free eulerian tour before Oksimets \cite{O97} proved Conjecture \ref{conj:eulerian} for $\ell=3$ with	 a sharp bound $d_3=6$.
Bensmail, Harutyunyan, Le and Thomass\'e recently verified Conjecture \ref{conj:eulerian} for 4-edge-connected eulerian graphs.
\begin{theorem}[\cite{BHLT15+}, Theorem 5.1] \label{theorem:simple 4eulerian}
For every positive integer $\ell$, there is an integer $d'_\ell$ such that every $4$-edge-connected eulerian graph
$G$ with minimum degree at least~$d'_\ell$ admits an $\ell$-step self-avoiding eulerian tour.
\end{theorem}

The main result of this paper is the following.

\begin{theorem} \label{theorem: main}
Conjecture \ref{conj:eulerian} is true.
\end{theorem}

Theorem \ref{theorem: main} gives an immediate corollary on edge-decomposition of graphs.
An \emph{edge-decomposition} of a graph $G$ consists of edge-disjoint subgraphs whose union is $G$.
Bar\'at and Thomassen in 2006 considered edge-decompositions of graphs into copies of a given tree and conjectured that, together with the necessary condition that $|E(H)|$ divides $|E(G)|$, large edge-connectivity may be an additional sufficient condition. 

\begin{conjecture} [Bar\'at--Thomassen conjecture, \cite{BT06}]\label{conjecture:barat-thomassen}
For any fixed tree $T$, there is an integer $c_T$ such that every 
$c_T$-edge-connected graph with number of edges divisible by $|E(T)|$ 
can be decomposed into subgraphs isomorphic to $T$.
\end{conjecture}

\noindent Conjecture~\ref{conjecture:barat-thomassen} was recently solved by Bensmail, Harutyunyan, Le, Merker and Thomass\'e in~\cite{BHLMT16+}.
For a summary of the progress towards the conjecture, we hence refer the interested reader to that paper.
Before that, the path case of the conjecture was verified by Botler, Mota, Oshiro and Wakabayashi in~\cite{BMOW14}, and then was improved by
Bensmail, Harutyunyan, Le, and Thomass\'e~\cite{BHLT15+}
that, for path-decompositions, high minimum degree is a sufficient condition provided the graph is 24-edge-connected. Very recently, Klimo\v{s}ov\'a  and Thomass\'e \cite{KT} reduced the edge-connectivity condition from 24 to 3, which is known to be sharp (see \cite{BHLT15+}).

Returning to $\ell$-step self-avoiding eulerian tours, by cutting the tour found by Theorem \ref{theorem: main} into paths of length $\ell$, we obtain the following variant of the path case of Bar\'at--Thomassen conjecture.

\begin{corollary} \label{theorem:eulerian2}
For every integer $\ell \geq 2$, there is an integer $d_\ell$ such that every eulerian graph
 with minimum degree at least~$d_\ell$ can be decomposed
into paths of length $\ell$ and possibly an additional path of length less than $\ell$.
\end{corollary}

Clearly, the theorems above cannot be extended to multigraphs; a multigraph consisting of two vertices linked by many edges is a counterexample. 
However, the main tool to prove Theorem \ref{theorem: main} is indeed a weak extension of Theorem \ref{theorem:simple 4eulerian} to multigraphs. Roughly speaking, we are happy if the eulerian tour behaves well on a given simple subgraph, not necessary on the whole multigraph.

\begin{theorem}\label{theorem: 4euler multi}
For every integer $\ell$, there is an integer $d_\ell$ such that for every 4-edge-connected eulerian multigraph $G$ with minimum degree at least $d_\ell$ and every simple subgraph $G'$ of $G$, the multigraph $G$ admits an eulerian tour of which every segment of length at most $\ell$ and consisting of only edges of $G'$ is a path.
\end{theorem}
This paper is organized as follows. We start by recalling some preliminary results in Section \ref{section:pre}. 
Then we use Theorem \ref{theorem: 4euler multi} as a black box to prove Theorem \ref{theorem: main} in Section \ref{section:main} before proving Theorem \ref{theorem: 4euler multi} in the last section.


\section{Preliminaries}\label{section:pre}
In this section we present all the auxiliary results necessary for our proof of
Theorem  \ref{theorem: main}. Given a multigraph $G$, let $V(G)$ and $E(G)$ denote its vertex and edge sets, respectively. 
For any subset $X$ of $V(G)$, let $G[X]$ denote the subgraph of $G$ induced by $X$. Given a
vertex $v$ of $G$, we denote by $d_G(v)$ 
the degree of $v$ in $G$. 
Given a subgraph $H=(V,F)$ of a multigraph $G = (V, E)$, we denote 
by $G \backslash H$
the multigraph $(V, E \backslash F)$.

We start by recalling the definition of cactus graphs. A connected loopless multigraph $G$ 
is a \emph{cactus} if every edge belongs to at most one cycle. 
The singleton graph is a cactus by convention.
Clearly, if a cactus is eulerian then every edge belongs to exactly one cycle. The following is a well-known property of cactus graphs.

\begin{proposition}\label{prop:thin}
There are at most two edge-disjoint paths between any two distinct vertices of a cactus.
\end{proposition}

We also recall three classical results. All of them are originally stated only for loopless multigraphs, but can be trivially generalized to multigraphs. Here we state their multigraph version. The first result due to de Werra (cf. \cite{dewerra}, Theorem 8.7), asserting that every multigraph has a balanced improper edge-coloring. 

\begin{proposition}\label{werra}
Let $G$ be a multigraph  and $k\ge 2$ be an integer. 
There is an improper edge-coloring of $G$ with $k$ colors 
 such that for every vertex $v$ and every 
pair of colors $i \neq j$, we have $|d_i(v)-d_j(v)|\le 4$, where $d_i(v)$ is the number of edges of color $i$ incident with $v$. 
\end{proposition}

The second is a result by Nash-Williams  \cite{NW60}
implying that every multigraph with high edge-connectivity admits a balanced orientation with
high arc-connectivity. In the following, a directed multigraph $D$ is \emph{$k$-arc-strong} if the removal of any set
of at most $k-1$ arcs leaves $D$ strongly-connected, and $d^+_D(v)$ and  $d^-_D(v)$ denote the outdegree and indegree of $v$ in $D$, respectively.

\begin{proposition} \label{prop: arc-strong}
Every $2k$-edge-connected multigraph has
an orientation $D$ such that $D$ is $k$-arc-strong and $|d^-_D(v) - d^+_D(v)| \leq 1$ for every vertex $v$.
\end{proposition}

The third result by Edmonds \cite{E73} expresses a condition for a directed multigraph to admit
many arc-disjoint rooted arborescences. In the statement, an \emph{out-arborescence}
of a directed multigraph $D$ refers to a rooted spanning tree $T$ of $D$ whose arcs are oriented
in such a way that the root has indegree~0, and every other vertex has indegree~1. 

\begin{proposition} \label{prop: disj-arbor}
A directed multigraph $D$
has $k$ arc-disjoint out-arborescences rooted at a given vertex $v$ if and only if for any vertex $u\ne v$, there are
$k$ arc-disjoint paths from $v$ to $u$.
\end{proposition}

We close this section with a result by Jackson (cf. \cite{J93}, Theorem 6.3). Given a loopless multigraph $G$, for every vertex $v$, let $E_v$ be the set of edges incident with $v$. A \emph{generalized transition system} $\mathcal{S}$
of $G$ is a set of functions $\{S_v\}_{v \in V(G)}$ with $S_v: E_v \to 2^{E_v}$
such that  $e_2 \in S_v(e_1)$ whenever  $e_1 \in S_v(e_2)$. 
We say that an eulerian tour $\mathcal{E}$ is \emph{compatible}
with $\mathcal{S}$ if for any two edges $e_1$ and $e_2$ such that $e_1 \in S_v(e_2)$ for some $v$, then
$e_1$ and $e_2$ are not consecutive edges of $\mathcal{E}$.

\begin{proposition}
\label{thm: Jackson}
Let $\mathcal{S}$ be a generalized transition system of a loopless eulerian multigraph $G$ such that $|S_v(e)|=0$ if $d(v)=2$ and $|S_v(e)| \leq d(v)/2 - 2$ if $d(v)\ge 4$ for any vertex $v$ and any edge $e$ incident with $v$. Then $G$ 
admits an eulerian tour compatible with $\mathcal{S}$.
\end{proposition}

\section{Proof of Theorem \ref{theorem: main}} 
\label{section:main}

The main idea of the proof of Theorem \ref{theorem: main} is as follows. We first partition the original graph $G$ into 4-edge-connected eulerian ``induced subgraphs''; these subgraphs are structurally linked by a big cactus. We then apply Theorem \ref{theorem: 4euler multi} to obtain a well-behaved eulerian tour of each subgraph, and finally connect these tours by the cactus to get an eulerian tour of $G$. 

Given a multigraph $G=(V,E)$, to \emph{contract} a set of vertices $X\subset V$, we remove all edges inside $X$, and then merge the vertices of $X$ to a new vertex $x$, where the edges incident with $x$ each corresponds to an edge incident with some $v\in X$. Note that if the sum of degrees of vertices of $X$ is even, then the degree of $x$ is even.

Let $G=(V,E)$ be an eulerian multigraph and $\mathcal{X}$ be a partition of $V$ into non-empty sets $X_1,X_2,...,X_k$ for some positive integer $k$. 
Let $M_\mathcal{X}$ be the loopless multigraph obtained from $G$ by contracting each $X_i$ to a new vertex $x_i$. 
Clearly, the degree of each $x_i$ of $M_\mathcal{X}$ is even. If $k\ge 2$, we have that $M_\mathcal{X}$ is connected since $G$ is connected, and hence $M_\mathcal{X}$ is eulerian.

Let us suppose for the moment that $M_\mathcal{X}$ is a cactus. 
Thus an edge $e$ of $M_\mathcal{X}$ belongs to exactly one cycle in $M_\mathcal{X}$. 
Let $e'$ be an edge of the same cycle and incident with $e$. 
We say that $\{e,e'\}$ is a \emph{pair at $x_i$}, where $x_i$ is some endpoint shared by $e$ and $e'$. 
Note that every edge belongs to exactly one pair at each of its endpoints, and hence belongs to exactly two pairs in total. 
Since each edge $e$ of $M_\mathcal{X}$ corresponds to an edge of $G$, we may use $e$ to denote both interchangeably.
For every pair $\{e,e'\}$ at some $x_i$, each edge has a unique endpoint in $X_i$, say $u$ and $u'$ respectively.
We create a new \emph{dummy edge}  $f=uu'$ \emph{associated with} the pair $\{e,e'\}$ (note that $f$ may be a loop). For every $1\le i\le k$, let $F_i$ be the edge set of $G[X_i]$ and $\FF_i$ be the set of all dummy edges on $X_i$, and let $G_i=(X_i,F_i\cup\FF_i)$. We say the multigraphs $G_1,...,G_k$ are \emph{inherited from} $\mathcal{X}$. Clearly, $d_{G_i}(v)=d_{G}(v)$ for every $v\in X_i$. 
The following lemma asserts that there is a partition such that inherited multigraphs are 4-edge-connected and eulerian, which are essential conditions to employ Theorem \ref{theorem: 4euler multi}.  For the sake of clarity, we do not consider edge-connectivity of multigraphs on a single vertex. 

\begin{lemma} \label{lem:separate}
Given an eulerian multigraph $G=(V,E)$, there exists a partition $\mathcal{X}$ of $V$ such that $M_\mathcal{X}$ is a cactus, and every $G_i$ inherited from $\mathcal{X}$ is either a single vertex with loops or a 4-edge-connected  eulerian multigraph. 
\end{lemma}

\begin{proof}
The proof is by induction on $|V|$. For the case $|V|=2$, let $V=\{u,v\}$. If $G$ has only two edges between $u$ and $v$, then $\mathcal{X}=\{\{u\},\{v\}\}$; otherwise, $\mathcal{X}=\{\{u,v\}\}$. The lemma holds true for $|V|=2$.

For the case $|V|>2$, if $G$ is 4-edge-connected, then $\mathcal{X}=\{V(G)\}$. 
Otherwise, $G$ contains an edge-cut of size 2, i.e. an edge-cut consisting of two edges. 
Consider an edge-cut partitioning $V$ into $X_1$ and $V'$ such that $|X_1|$ is minimum among all possible edge-cuts of size 2. 
Let call two edges of the cut $u_1v_1$ and $u_2v_2$, where $u_1,u_2\in X_1$ and $v_1,v_2\in{V'}$.  
We create two dummy edges $f=u_1u_2$ and $f'=v_1v_2$.
Let $\FF_1=\{f\}$, and $F_1=E(G[X_1])$.
Let $G_1=(X_1,F_1\cup\FF_1)$ and $G'=(V',E(G[V'])\cup\{f'\})$. 
There are at least two edge-disjoint paths in $G$ between any two distinct vertices of $X_1$. If both paths contains vertices of $V'$, then the edge-cut must has size at least 4, a contradiction. Therefore there is a path in $G[X_1]$ between any two distinct vertices of $X_1$.
Thus if $|X_1|>1$ then $G_1$ is connected, and hence is eulerian since the degree of every vertex of $G_1$ is even. Similarly, $G'$ is eulerian.

Suppose that $G_1$ contains an edge-cut of size 2 partitioning $X_1$ into $X'_1$ and $X''_1$. 
If $u_1$ and $u_2$ are in the same partition, say $X'_1$, then that edge-cut is also an edge-cut of $G$ partitioning $V$ into $X_1''$ and $V'\cup X_1'$, which contradicts the minimality of $|X_1|$. 
If $u_1\in X_1'$ and $u_2\in X_1''$ then that edge-cut consists of $f$ and another edge, say $e$. 
Then $\{e,u_1v_1\}$ is an edge-cut of $G$ partitioning $V$ into $X_1'$ and $V'\cup X_1''$, a contradiction again. 
It follows that $G_1$ contains no edge-cut of size 2, and so is 4-edge-connected.

Applying induction hypothesis to the eulerian multigraph $G'$ gives a partition of $V'$ into $\mathcal{X'}=\{X_2,...,X_{k}\}$ such that $M_{\mathcal{X}'}$ and $G_2,...,G_k$ inherited from $\mathcal{X'}$ satisfy Lemma \ref{lem:separate}. Let $x_i \in M_{\mathcal{X}'}$ corresponds to $X_i$ for every $2\le i\le k$.
Set $\mathcal{X}=\mathcal{X'}\cup\{X_1\}$  and construct $M_{\mathcal{X}}$ as follows:
\begin{enumerate}[label=(\alph*)]
\item \label{enum1} If $v_1,v_2\in G_i$ for some $i$, then $M_{\mathcal{X}}$ is obtained from $M_{\mathcal{X}'}$ by adding $x_1$ and two parallel edges $x_1x_i$, corresponding to edges $u_1v_1$ and $u_2v_2$ of $G$. 
Hence there is only one pair at $x_1$: $ \{u_1v_1,u_2v_2\}$, and $f$ is its associated dummy edge.
There is one more pair at $x_i$ in $M_{\mathcal{X}}$ comparing with $x_i$ in $M_{\mathcal{X}'}$: $ \{v_1u_1,v_2u_2\}$, and $f'$ is its associated dummy edge.

\item \label{enum2} Otherwise, $v_1\in G_i$ and $v_2\in G_j$ for some $i\ne j$. There must be an edge $x_ix_j$ in $M_{\mathcal{X}'}$ corresponding to  $f'$ in $G'$. 
We obtain $M_{\mathcal{X}}$ from $M_{\mathcal{X}'}$ by adding vertex $x_1$, edge $x_1x_i$ corresponding to $u_1v_1$ and edge $x_1x_j$ corresponding to $u_2v_2$ together with deleting the edge $x_ix_j$ corresponding to $f'$. 
There is only one pair at $x_1$: $ \{u_1v_1,u_2v_2\}$, and $f$ is its associated dummy edge. 
The set of pairs at $x_i$ (res. $x_j$) of $M_{\mathcal{X}}$ are identical to the set of pairs at $x_i$ (res. $x_j$) of $M_{\mathcal{X}'}$, except that $v_1u_1$  (res. $v_2u_2$) replaces $f'$ in some pair at $x_i$ (res. at $x_j)$.
\end{enumerate}

The multigraphs $G_2,...,G_k$ inherited from $\mathcal{X}$ in this construction are identical to the multigraphs $G_2,...,G_k$ inherited from $\mathcal{X'}$.
By induction hypothesis, for every $i\ge 2$, if $G_i$ has more than one vertex then it is 4-edge-connected and eulerian.
Note that $x_1$ has degree 2, and $M_{\mathcal{X}'}$ is a cactus, then so is $M_{\mathcal{X}}$. This proves the lemma.
\end{proof}

Given an eulerian tour $\EE$ of $G$ and a subset $X$ of $V$, a segment $v_1v_2...v_r$ ($r \ge 3$) of $\EE$ is an \emph{$X$-boomerang} if $v_1,v_r\in X$ and $v_2,...,v_{r-1}\notin X$.
A \emph{projection} of $\EE$ on $X$ is an eulerian tour $\EE_X$ obtained from $\EE$ by replacing every $X$-boomerang,  say $v_1v_2...v_r$, by a dummy edge (possibly a loop) between $v_1$ and $v_r$. If $\EE_X$ is a projection of $\EE$, we say $\EE$ and $\EE_X$ are \emph{compatible}. 

Let $G$ be an eulerian multigraph and $\mathcal{X}$ be a partition of $G$ together with $M_\mathcal{X}$ and inherited $G_1,...,G_k$ obtained by the algorithm in the proof of Lemma \ref{lem:separate}.   For every $i$, let $\EE_i$ be an arbitrary eulerian tour of $G_i$. 

\begin{claim}\label{lem:compatible}
There exists an eulerian tour $\EE$ of $G$ compatible with all $\EE_i$. Furthermore, for every pair $\{e,e'\}$ at some $x_i$, there is an $X_i$-boomerang of $\EE$ starting and ending by $e$ and $e'$.
\end{claim}

\begin{proof}
We reuse all notations in the proof of Lemma \ref{lem:separate} and proceed by induction on $k$. 
The claim clearly holds true for $k=1$. 
For $k>1$, recall that by the algorithm in the proof of Lemma \ref{lem:separate}, the eulerian multigraph  $G'$ has $k-1$ inherited multigraphs  identical to $G_2,...,G_k$ of $G$. 
Hence applying induction hypothesis of Claim \ref{lem:compatible} to $G'$ results in an eulerian tour $\EE'$ of $G'$ compatible with all $\EE_i, i\ge 2$, and for every pair $\{e,e'\}$ at some $x_i,i\ge 2$, there is an $X_i$-boomerang of $\EE'$ starting and ending by $e$ and $e'$. 
Note that in both cases \ref{enum1} and \ref{enum2}, the only pair at $x_1$ is $\{u_1v_1,u_2v_2\}$ associated with $f$. 
Let $W_1$ be the walk obtained from $\EE_1$ by removing $f$, and $\EE$ be the eulerian tour on $G$ obtained from $\EE'$ by replacing $f'$ by the segment $v_1u_1W_1u_2v_2$.
It is straightforward that, in both cases \ref{enum1} and \ref{enum2}, the tour $\EE$ satisfies Claim \ref{lem:compatible}.
\end{proof}

Let $\{e,e'\}$ be a pair at some $x_i$, and $W$ be the $X_i$-boomerang of $\EE$ starting and ending by $e$ and $e'$. Let $\overline{W}$ be the segment obtained from $\EE$ by removing $W$.

\begin{claim}\label{lem:farewell}
If $W$ visits a vertex $v\notin X_i$, then $\overline{W}$ does not visit $v$.
\end{claim}

\begin{proof}
Suppose that the claim was false. Let $v\in X_j$ for some $j\ne i$. Contracting every $X_i$ to $x_i$ naturally yields from $W$ and $\overline{W}$ two edge-disjoint walks $W_{\mathcal{X}}$ and $\overline{W}_{\mathcal{X}}$ in $M_{\mathcal{X}}$, respectively.
By following $W_{\mathcal{X}}$ from $x_i$ to $x_j$ and return to $x_i$, and then following $\overline{W}_{\mathcal{X}}$ to $x_j$, we obtain three edge-disjoint walks between $x_i$ and $x_j$, contrary to Proposition \ref{prop:thin}.
\end{proof}

\begin{claim}\label{lem:refarewell}
If $G$ has minimum degree $d$, then whenever $\EE$ leaves $X_i$, it takes at least $d$ steps to return to $X_i$. 
\end{claim}
\begin{proof}
The claim is equivalent to that every $X_i$-boomerang $W$ has length at least $d$. Suppose that $W$ visits vertex $v\notin X_i$. 
By Claim \ref{lem:farewell}, $W$ must contains all edges incident with $v$, 
and hence has length at least $d$.
\end{proof}

  We are ready to prove the main theorem.

\begin{proof}[Proof of Theorem \ref{theorem: main}] 
Let $G$ be an eulerian graph with minimum degree at least $d_\ell$, the constant of Theorem \ref{theorem: 4euler multi}.
There is a partition $\mathcal{X}=\{X_1,...,X_k\}$ of $V(G)$ together with inherited multigraphs $G_1,...,G_k$ satisfying Lemma \ref{lem:separate}.

If $G_i$ consists of only one vertex and some loops, let $\EE_i$ be an arbitrary eulerian tour of $G_i$. Otherwise, Lemma \ref{lem:separate} asserts that $G_i$ is eulerian, 4-edge-connected, and $d_{G_i}(v)=d_G(v)\ge d_\ell$ for any $v\in X_i$.
Also note that $G[X_i]$ is a simple subgraph of $G_i$.
We thus get, by Theorem \ref{theorem: 4euler multi}, an eulerian tour $\EE_i$ of $G_i$ of which every segment of length at most $\ell$ and containing only edges of $G[X_i]$ is a path. 
Claim \ref{lem:compatible} gives an eulerian tour $\EE$ of $G$ compatible with all $\EE_i$. 

The proof is completed by showing that every segment $W$ of length at most $\ell$ of $\EE$ is a path. Suppose that $W=W_1e_1W_2e_2...e_{t-1}W_t$, where each $W_s$ (possibly of length 0) contains only vertices of some $X_{i_s}$, and $e_s$ is an edge between two distinct sets $X_{i_s}$ and $X_{i_{s+1}}$. 
By Claim \ref{lem:refarewell}, whenever $\EE$ leaves some $X_{i_s}$, it takes at least $d_\ell>\ell$ steps to return to $X_{i_s}$, while the length of $W$ is at most $\ell$. Therefore $X_{i_s}\ne X_{i_{r}}$ for every $s\ne r$. Because $\EE$ is compatible with $\EE_{i_s}$, and $W_s$ contains only vertices of $X_{i_s}$, we have that $W_{s}$ is a segment of $\EE_{i_s}$.
Since $W_s\subseteq G[X_{i_s}]$ and has length at most $\ell$, it is a path by Theorem \ref{theorem: 4euler multi}.
This means that $W$ is a path, and the proof is complete.
\end{proof}


\section{Proof of Theorem \ref{theorem: 4euler multi}} 
\label{section:relaxation}

\subsection{Path-collections} \label{section:fractions}

We first recall some notions and results in \cite{BHLT15+}.
Let $G=(V,E)$ be a loopless multigraph. A {\em path-collection} $\PP$ \textit{on} $G$ is a set of edge-disjoint paths of $G$.
We denote by $U_\PP=(V,E')$ the multigraph where $E'$ is the set of edges of paths in ${\cal P}$. 
If $U_\PP=G$ then $\PP$ is said to be a \textit{path-decomposition} of $G$.
For convenience, from now on, we say \emph{collection} for path-collection and \emph{decomposition} for path-decomposition.

Let us denote by $H_\PP=(V,E'')$ the multigraph where each edge $uv\in E''$ corresponds to a path between $u$ and $v$ in $\PP$ (if $\PP$ contains several paths from $u$ to $v$, we have as many 
edges $uv\in E''$). 
The {\em degree} 
of a vertex $v$ in $\PP$, denoted $d_\PP(v)$, is the degree (with multiplicity) of $v$ in $H_\PP$, which is also the number of paths in $\PP$ with endpoint $v$.

Two edge-disjoint paths of $G$ sharing an endpoint $v$ are \textit{conflicting} if they also intersect at some vertex different from $v$.
Equivalently, we say that two paths of $\PP$ issued from the same vertex are \emph{conflicting}
if the corresponding paths in $U_\PP$ are conflicting.
In general, the paths of a collection can pairwise intersect, and hence we would like to measure
how much.
For every vertex $v\in V$, let $\PP(v)$ be the set of paths in $\mathcal{P}$ containing $v$ as an endpoint. 
	The {\em conflict ratio} of $v$ is $$\mbox {conf}_{\PP}(v):=\frac{\max_{w\ne v} \big| \{P\in \PP(v) : w \in P\}\big|}{d_\PP(v)}.$$
We denote the \emph{conflict ratio} of $\PP$ by  ${\rm conf}(\PP):=\max_{v}\mbox {conf}_{\PP}(v).$
We always have ${\rm conf}(\PP) \leq 1$ since $|\PP(v)|=d_\PP(v)$. 

Suppose that we have a decomposition $\PP$ of an eulerian graph $G$ with all paths of length at least $\ell$. Then just by concatenating the paths arbitrarily, we obtain a decomposition of $G$ into several circuits since $G$ is eulerian. If every two consecutive paths (i.e., they are concatenated) are non-conflicting, then all circuits 
are  $\ell$-step self-avoiding. Theorem \ref{ll1} provides a low conflicting decomposition for this purpose.
In order to obtain an $\ell$-step self-avoiding eulerian tour, it is necessary that the process of concatenating returns a single circuit; this is taken care by Lemma \ref{lemma:13-tree}.

\begin{theorem}[\cite{BHLT15+}, Theorem 3.4] \label{ll1}
Let $\ell$ be a positive integer, and $\varepsilon>0$ sufficiently small. 
There is an integer $L_{\ell,\varepsilon}$ such that for every graph $G$ with minimum degree at least $L_{\ell,\varepsilon}$, 
there is a decomposition $\PP$ of $G$ satisfying:
\begin{itemize}
\item The length of every path of $\PP$ is either $\ell$ or $\ell+1$.
	\item ${\rm conf}(\PP) \leq 1/4(\ell + 10)$.
	\item $(1-\varepsilon) d_G(v) \leq   \ell d_{\PP}(v)
	\leq (1+\varepsilon) d_G(v)$ for every vertex $v$.
\end{itemize}
\end{theorem}

\begin{lemma}[\cite{BHLT15+}, Lemma 4.1] \label{lemma:13-tree}
Every 2-edge-connected loopless multigraph $G$
has a collection $\PP$ such that the length of every path in $\PP$ is either $1$ or $2$, and $H_\PP$ is a subcubic tree spanning $V(G)$.
\end{lemma}

\subsection{$F$-path-collections}

Given a multigraph $G=(V,E)$ and a subgraph $G'=(V,F)$ satisfying the hypotheses of Theorem \ref{theorem: 4euler multi}, the goal is to find an eulerian tour 
$\EE$ of $G$ such that every segment of $\EE$ of length at most $\ell$ and consisting of only edges of $F$ is a path. To this end, we introduce a relaxation of path, called $F$-path, to depict the characteristics of segments of the tour.
Let $G=(V,E)$ be a multigraph and $F$ be a subset of $E$.
A walk $W$ in $G$ is called an \emph{$F$-path} if every subwalk of $W$ containing only edges of $F$ is a path. An $F$-path $W$ is \emph{covered} if all edges of $W$ belong to $F$, and is \emph{uncovered} otherwise.
It is immediate that a covered $F$-path is a path. 

An \emph{$F$-collection} $\PP$ on $G$ is a set of edge-disjoint $F$-paths of $G$.
We denote by $U_\PP=(V,E')$ the multigraph where $E'$ is the set of edges of $F$-paths in ${\cal P}$. 
If $U_\PP=G$, then $\PP$ is called an \textit{$F$-decomposition} of $G$.
We denote by $H_\PP=(V,E'')$ the multigraph where each edge (possibly a loop) $uv\in E''$ corresponds to an $F$-path between $u$ and $v$ in $\PP$. The {\em degree} 
of a vertex $v$ in $\PP$, denoted $d_\PP(v)$, is the degree (with multiplicity, and a loop contributes two) of $v$ in $H_\PP$.

Given an $F$-path $P=ve_1v_1...e_{t}v_{t}$, the \emph{ray of $P$ from }$v$, denoted by $P_{v|F}$, is the longest subwalk $ve_1v_1...e_{s}v_{s}$ (possibly of length 0) of $P$ such that $e_1,...,e_{s}\in F$. There are several remarks. First, every ray is a path.
Second, each $F$-path $P$ has exactly two rays; these rays are identical to $P$ if $P$ is covered, and are edge-disjoint if $P$ is uncovered. 
Third, if $P$ is a \emph{closed} (obviously uncovered) $F$-path from $v$ to $v$, then both of its rays are from $v$.
We now would like to measure the conflict between two rays.
We first agree that two rays of the same $F$-path do \emph{not conflict} each other, even if they may intersect at some vertex.
Two rays $P_{v|F}$ and $P'_{v|F}$ (with $P\ne P'$) issued from some vertex $v$ are 
 \emph{conflicting} if $P_{v|F}$ and $P'_{v|F}$ also intersect at some vertex different from $v$.
For every $v\in V$, let ${\cal P}(v)$ be the set of $F$-paths in $\mathcal{P}$ containing $v$ as an endpoint, and ${\cal P}(v|F)$ be the set of rays from $v$ of $F$-paths in $\mathcal{P}$, where a closed $F$-path with endpoint $v$ contributes two rays.
We define the \emph{conflict ratio} of $v$ in $\PP$ as
$$ \mbox {conf}_{\PP}(v|F):=\frac{\max_{w\ne v} \big| \{P_{v|F}\in {\cal P}(v|F) : w \in P_{v|F}\}\big|}{d_\PP(v)}.$$
We denote the \emph{conflict ratio} of $\PP$ by $  \mbox {conf}(\PP|F):=
\max_v\mbox {conf}_{\PP}(v|F)$. 
We always have ${\rm conf}(\PP|F) \leq 1$ since $\big|{\cal P}(v|F)\big|=d_\PP(v)$.

The the proof of Theorem \ref{theorem: 4euler multi} is similar to the proof of Theorem \ref{theorem:simple 4eulerian} but more involved.
Let us first prove an extension of Theorem \ref{ll1} to $F$-decompositions. By saying a ray of $\PP$, we mean a ray of some $F$-path of $\PP$.

\begin{lemma}\label{lemma: highmin multi}
Let $\ell$ be a positive integer, and $\varepsilon>0$ sufficiently small. 
There is an integer $L'_{\ell,\varepsilon}$ such that for every multigraph $G$ with minimum degree at least $L'_{\ell,\varepsilon}$ and every simple subgraph $(V,F)$ of $G$, 
there is an $F$-decomposition $\PP$ of $G$ satisfying:
\begin{itemize}
	\item Every ray of $\PP$ has length at most $\ell+1$.
	\item Every covered $F$-path of $\PP$ has length at least $\ell$.
	\item ${\rm conf}(\PP|F) \leq 1/4(\ell + 9)$.
	\item $(1-\varepsilon) d_G(v) \leq   \ell d_{\PP}(v)
	\leq (1+2\varepsilon) d_G(v)$ for every vertex $v$.
\end{itemize}
\end{lemma}

\begin{proof}
Set $L'_{\ell,\varepsilon}=\max(L_{\ell,\varepsilon},2\ell/\varepsilon)$, where $L_{\ell,\varepsilon}$ is the constant of Theorem \ref{ll1}. 
We call all edges of $\FF=E\backslash F$ \emph{dummy} (note that a dummy edge may be a loop).
The main idea is to replace every dummy edge by a pair of edges linking endpoints of the dummy edge to a big clique in order to obtain a simple graph to apply Theorem \ref{ll1}. 
For every dummy edge $e=v_{e,1}v_{e,2}$, we create a set of $L_{\ell,\varepsilon}+1$ new vertices $X_e=\{x_{e,1},...,x_{e,L_{\ell,\varepsilon}+1}\}$.
Let $E_e=\{x_{e,i}x_{e,j}: i\ne j\}\cup \{v_{e,1}x_{e,1},v_{e,2}x_{e,2}\}$. 
Let $G'$ be the multigraph with vertex set $\bigcup_{e\in \FF} X_e\cup V$ and edge set $E'=\bigcup_{e\in\FF}E_e\cup F$. 
It is immediate that $G'$ is simple and $d_G(v)=d_{G'}(v)$ for every $v\in V$, and so $G'$ has minimum degree at least $L$. 
Therefore $G'$ admits a decomposition $\PP'$ satisfying Theorem \ref{ll1}. 

For every dummy edge $e$ and every $i=1,2$, let $P'_{e,i}$ be the path of $\PP'$ containing $v_{e,i}x_{e,i}$. 
We denote by $P_{i,j}$ the longest possible subwalk of $P'_{e,i}$ such that $P'_{e,i}=...x_{e,i}v_{e,i}P_{e,i}...$ and all vertices of  $P_{e,i}$ belong to $V$. 
If $P_{e,i}$ reach the end of $P'_{e,i}$, we call $P_{e,i}$ an \emph{end-segment}; otherwise, we call it a \emph{middle-segment}. 
The reader may see here the similarity between end-segments and rays.
Clearly, if $P_{e,i}$ is a middle-segment, then $P'_{e,i}=...x_{e,i}v_{e,i}P_{e,i}
v_{e',j}x_{e',j}...$ for some dummy edge $e'$ and $j\in \{1,2\}$ since $P'_{e,i}$ leaves $V$ right after finishing $P_{e,i}$.
Note also that the lengths of end-segments and middle-segments are at most $\ell+1$ and possibly 0.

For every dummy edge $e$ and every $i=1,2$, we remove $X_e$ and $E_e$, and concatenate $P_{e,i}$ with $e$ at $v_{e,i}$. After this process, we obtain a family of walks, each lies in one of the following types:

\begin{itemize}
\item An uncovered $F$-path $P=P_1e_1P_2...e_{t-1}P_{t}$ with dummy edges $e_1,...,e_{t-1}$, end-segments $P_1$ and $P_{t}$, and middle-segments $P_2,...,P_{t-1}$. Note that the two end-segments are the rays of this uncovered $F$-path. Let $\PP_1$ be the set of all these uncovered $F$-paths together with all paths of $\PP'$ containing only vertices of $V$.

\item A circuit without endpoint, consisting of middle-segments alternate with dummy edges but no end-segments. Let $\PP_2$ be the set of all these circuits.
\end{itemize} 
 
Note that $\PP_1$ is a $F$-collection of $G$, and every edge of $G$ belongs to exactly one $F$-path $\PP_1$ or one circuit of $\PP_2$.
The method of concatenating ensures that for every $v\in V$, the number of rays from $v$ in $\PP_1$ is equal to number of paths with endpoint $v$ in $\PP'$.  
This gives $d_{\PP_1}(v)=d_{\PP'}(v)$. 
Besides, each ray of $\PP_1$ is the end-segment of some path of $\PP'$. 
Therefore two rays of $\PP_1$ are conflicting only if their corresponding paths in $\PP'$ are conflicting. Thus all of the following hold true:
\begin{itemize}
	\item Every ray of $\PP_1$ has length at most $\ell+1$, since it is either a path or an end-segment of some path of $\PP'$.
	\item Every covered $F$-path of $\PP_1$ has length at least $\ell$, since it is either a path of $\PP'$.
	\item ${\rm conf}_{\PP_1}(v|F) \leq {\rm conf}_{\PP'}(v) \leq 1/4(\ell + 10)$ for every vertex $v$.
	\item $(1-\varepsilon) d_G(v) \leq   \ell d_{\PP_1}(v)
	\leq (1+\varepsilon) d_G(v)$ for every $v$ since $d_{\PP_1}(v) = d_{\PP'}(v) $.
\end{itemize}

We now turn our attention to $\PP_2$. Every circuit $C\in\PP_2$ contains at least one dummy edge. We \emph{associate} $C$ with some vertex $v$ such that $v$ is the endpoint of some dummy edge of $C$. For every $v\in V$, let $C_1,...,C_t$ be the circuits (if any) associated with $v$, where every $C_s=ve_sW_sv$ with dummy edge $ e_s$. 
Let $\hat{P}_v=ve_1W_1 ve_2W_2...v e_tW_tv$ be the walk starting and ending at $v$ obtained by concatenating all $C_s$ in that fashion. 
Clearly, $\hat{P}_v$ is an uncovered $F$-path, of which one ray is $v$ (length 0) and another ray is $W_t$, a middle-segment of length at most $\ell+1$. 
Note that for every $v$, we have at most one such $\hat{P}_v$.
Let $\hat{\PP_2}=\{\hat{P}_v:v\in V\}$. Then $\hat{\PP_2}$ is an $F$-collection of $G$ and $U_{\PP_1} \cup U_{\hat{\PP_2}}=G$.
Hence $\PP=\PP_1\cup \hat{\PP_2}$ is an $F$-decomposition of $G$. 
Then every ray of $\PP$ has length at most $\ell+1$, and every covered $F$-path of $\PP$ has length at least $\ell$.

For every $v$, the number of rays from $v$ of $\PP$ is at most the number of rays from $v$ of $\PP_1$ plus two (two rays of $\hat{P}_v$ if it exists). Hence $d_{\PP_1}(v)\le d_\PP(v)\le d_{\PP_1}(v)+2$, and so by definition of conflict ratio, we have 
\begin{align*}
{\rm conf}_{\PP}(v|F) &\leq  \frac{d_{\PP_1}(v){\rm conf}_{\PP_1}(v|F)+2}{d_\PP(v)} \\
&\leq {\rm conf}_{\PP_1}(v|F)+ \frac{2}{d_\PP(v)}\\
&\frac{1}{4(\ell + 10)}+ \frac{2}{d_\PP(v)}\\
&\le  \frac{1}{4(\ell + 9)}. 
\end{align*} 

Finally, we have $(1-\varepsilon) d_G(v) 
\le \ell d_{\PP_1}(v) \leq \ell d_{\PP}(v)$. And since $L_{\ell,\varepsilon}'\ge 2\ell/\varepsilon$, we have 
$
\ell d_{\PP}(v)\le \ell (d_{\PP}(v) + 2) 
\le (1+2\varepsilon) d_G(v)$.
The proof is complete.
\end{proof}

Lemma \ref{lemma: highmin multi} gives us a good $F$-decomposition $\PP$ of $G$. We wish to concatenate the $F$-paths of $\PP$ to an eulerian tour. 
If $H_\PP$ has an eulerian tour, we naturally obtain an eulerian tour of $G$ by replacing each edge of $H_\PP$ by its corresponding $F$-path of $\PP$. Hence the goal is achieving the connectivity of $H_\PP$, which immediately yields eulerianity thank to the fact that every vertex of $H_\PP$ has even degree.

\begin{lemma}\label{lemma: 4ec multi} Under the same hypotheses of Lemma \ref{lemma: highmin multi} except that $G$ is 4-edge-connected with minimum degree at least $100 \ell L'_{\ell,\varepsilon}$,
there is an $F$-decomposition $\PP$ of $G$ satisfying:
\begin{itemize}
	\item Every ray of $\PP$ has length at most $\ell+3$.
	\item Every covered $F$-path of $\PP$ has length at least $\ell$.

\item   ${\rm conf}(\PP|F) \leq 1/2(\ell + 9)$.

\item $H_\PP$ is eulerian and spans $V(G)$.

\end{itemize}
\end{lemma}

\begin{proof}
Let us first outline the proof.
We wish to obtain connectivity of $\PP$. To this end, we decompose $G$ into a collection $\PP_0$ satisfying Lemma \ref{lemma:13-tree} and two $F$-collections $\PP_1$ and $\PP_2$ satisfying Lemma \ref{lemma: highmin multi}. 
Then we use $\PP_0$, which contains only paths of short length, to tweak $F$-paths of $\PP_1$ to make $H_{\PP_1}$ connected. Finally, we merge $\PP_1$ with $\PP_2$ to obtain $\PP$, which inherits connectivity from $\PP_1$ and low conflict ratio from $\PP_2$.

Because $G$ is 4-edge-connected, by Proposition \ref{prop: arc-strong}, there is an orientation $D$ of $G$
such that $D$ is $2$-arc-strong and $|d^{+}_D(v)-d^{-}_D(v)|\le 1$ for every $v$.
Applying Proposition \ref{prop: disj-arbor} to $D$ with an arbitrary vertex $z$ gives us two arc-disjoint out-arborescences, $T_1, T_2$, rooted at $z$.  
Each vertex $v$ has indegree at most~$1$ in each $T_i$ ($z$ has indegree 0). This gives $d_{T_1 \cup T_2}(v)\le d_D^+(v)+2\le  d_G(v)/2+3$ for every vertex $v$ since $|d^{+}_D(v)-d^{-}_D(v)|\le 1$. Because $T_1\cup T_2$ is loopless and 2-edge-connected, we obtain a collection $\PP_0$ on $T_1\cup T_2$ satisfying Lemma \ref{lemma:13-tree}.

Let $G'=G\backslash U_{\PP_0}$. 
Then $d_{U_{\PP_0}}(v)\le d_{T_1 \cup T_2}(v)\le d_G(v)/2+3$, and so $G'$ has minimum degree at least $100\ell L_{\ell,\varepsilon}'/2-3\ge 48\ell L_{\ell,\varepsilon}'$. 
By Proposition \ref{werra}, $G'$ has an improper coloring by $45\ell$ colors
such that $|d_i(v)-d_j(v)|\le 4$ for every vertex $v$ and every 
pair of colors $i \neq j$. Let $G_1$ be the subgraph of $G'$ with edge set of the first color, and $G_2=G'\backslash G_1$. Thus
 $$d_{G_1}(v)\le \frac{1}{45\ell-1}d_{G_2}(v)+4\le \frac{d_{G_2}(v)}{40\ell}.$$
The minimum degrees of both $G_1$ and $G_2$ are at least $48\ell L_{\ell,\varepsilon}'/45\ell-4\ge L_{\ell,\varepsilon}'$. Therefore 
there are $F$-decompositions $\PP_1$ of $G_1$ and $\PP_2$ of $G_2$, both satisfying Lemma \ref{lemma: highmin multi}. 
Hence
$$
 d_{\PP_1}(v) 
\le 
\frac{1+2\varepsilon}{\ell}d_{G_1}(v)
\le  \frac{1+2\varepsilon}{40\ell^2} d_{G_2}(v)
\le  \frac{1+2\varepsilon}{40\ell(1-\varepsilon)}d_{\PP_2}(v),
$$ for every vertex $v$. 
Set $\varepsilon$ small enough such that for every $v$,

\begin{equation}\label{eq: H1}
d_{\PP_1}(v) 
\leq \frac{1}{4(\ell+9)} d_{\PP_2}(v)-3.
\end{equation}

We now turn our attention to the collection $\PP_0$ and the subcubic spanning tree $H_{\PP_0}$. 
Let us consider $H_{\PP_0}$ as a tree rooted at an arbitrary vertex $z$. In the following claim, we collect two private $F$-paths in $\PP_1$ for each path in $\PP_0$ for the process of concatenating later on.

\begin{claim}
For every path $P\in \PP_0$ with endpoints say $u,v$ where $v$ is the parent of $u$ in $H_{\PP_0}$, there are two $F$-paths of $\mathcal{P}_1(v)$, named $g_1(P)$ and $g_2(P)$, such that their rays from $v$ do not conflict with $P$ (if $g_i(P)$ is closed, one of its rays satisfying that condition is sufficient). Furthermore, $g_i(P)\ne g_j(P')$ for any $(i,P)\ne (j,P')$.

\end{claim}
\begin{proof}
We first apply Proposition \ref{prop: arc-strong} to have an orientation $D$ of $H_{\PP_1}$ such that $|d^-_{D}(v)-d^+_{D}(v)|\le 1$. This orientation yields a natural orientation of $F$-paths of $\PP_1$. We denote by $\mathcal{P}^+_1(v|F)$ the set of rays from $v$ of $\PP$ corresponding to $D$. Note that each closed $F$-path at $v$ contributes exactly one ray to $\mathcal{P}^+_1(v|F)$. This gives $|\mathcal{P}^+_1(v|F)|\ge d_{\PP_1}(v)/2-1$.

Since $H_{\PP_0}$ is subcubic, there are at most 3 paths of $\PP_0$ with endpoint $v$, say $P_s$ for $1\le s\le 3$. 
Note that each $P_{s}$ has length at most 2, and so they are incident with at most $6$ vertices except $v$ in total. Recall that ${\rm conf}_{\PP_1}(v|F) \leq 1/4(\ell + 9)$. For each vertex $w$ among these 6 possible vertices, we have
$$\Big| \{P_{v|F}\in {\cal P}_1(v|F) : w \in P_{v|F}\}\Big|\le \frac{d_{\PP_1}(v)}{4(\ell+9)}\le \frac{2|\mathcal{P}^+_1(v|F)|+2}{4(\ell+9)}\le \frac{|\mathcal{P}^+_1(v|F)|}{12}.$$
Hence in total there are at most $|\mathcal{P}^+_1(v|F)|/2$ rays of $\mathcal{P}^+_1(v|F)$ conflicting with some $P_s$.
This guarantees that there are at least half of rays in $\mathcal{P}^+_1(v)$ non-conflicting with all $P_s$. We just pick 6 rays among them, and name the $F$-paths of these rays $g_i(P_s)$ arbitrarily (these $F$-paths are clearly pairwise distinct).
Note also that $\mathcal{P}^+_1(v)\cap \mathcal{P}^+_1(v')=\emptyset$ for any $v\ne v'$, so $g_i(P)\ne g_j(P')$ for any $(i,P)\ne (j,P')$.
\end{proof}

We can now obtain the connectivity of $H_1$ by concatenating each $P$ of $\PP_0$ to either $g_1(P)$ or $g_2(P)$.
Let us call $\mathcal{T}$ a rooted tree on vertex set $\{Y_1,Y_2,...,Y_t\}$, where $\{Y_1,Y_2,...,Y_t\}$ is some partition of $V$ 
with the following properties:

\begin{enumerate}[label=(\Alph*)]

\item For every edge $Y_iY_j$ of $\mathcal{T}$, there is
a corresponding path $v_i...v_j \in \PP_0$, where 
$v_i \in Y_i$ and $v_j \in Y_j$.

\item For every $Y_i$, there is an $F$-collection $\mathcal{R}_i$ such that $H_{\mathcal{R}_i}$ is connected and spans $Y_i$, and each $F$-path in $\mathcal{R}_i$ is either $g_1(P)$ or the concatenation of $P$ and $g_1(P)$ for some $P\in \PP_0$ (if $Y_i$ contains a single vertex then $\mathcal{R}_i$ is empty). 

\end{enumerate}

Such structured-tree $\mathcal{T}$ clearly exists by choosing $\mathcal{T}$ equal to $H_{\PP_0}$ rooted at $z$, in which each
$Y_i$ contains a single vertex, and each $\mathcal{R}_i$ is empty.
Our goal is to repeatedly merge vertices of $\mathcal{T}$ until $\mathcal{T}$ is the singleton graph, which completes the process of concatenating. 
We consider a leaf $Y_i$ of $\mathcal{T}$ with parent $Y_j$, corresponding to path $P=v_i...v_j$ of  $\PP_0$ with $v_i\in Y_i$ and $v_j\in Y_j$. Suppose that $g_1(P)=v_j...y$ and $ g_2(P)=v_j...z$.  
\begin{itemize}
\item 
If $y\in Y_k$ for some $ k\neq i$,
we concatenate $P$ and $g_1(P)$ at $v_j$ and get a $F$-path $P^*$.
Then we merge $Y_i$ into $Y_k$ to form 
new set $Y_{ik}$ (inheriting the position of $Y_k$ in tree $\mathcal{T}$). Let  $\mathcal{R}_{ik}=\mathcal{R}_{i}\cup \mathcal{R}_{k}\cup \{P^*\}$. Since $P^*$ connects two vertices of $\mathcal{R}_i$ and $\mathcal{R}_k$, we have that $H_{\mathcal{R}_{i,k}}$ is connected and spans $Y_{ik}$.

\item If $y\in Y_i$, we merge $Y_i$ to $Y_j$ to form new set $Y_{ij}$ (inheriting the position of $Y_j$ in tree $\mathcal{T}$). Set $\mathcal{R}_{ij}=\mathcal{R}_{i}\cup \mathcal{R}_{j}\cup \{g_1(P)\}$. Since $g_1(P)$ connects two vertices of $\mathcal{R}_i$ and $\mathcal{R}_j$, we have that $H_{\mathcal{R}_{i,k}}$ is connected and spans $Y_{ij}$.
We also 
concatenate $P$ with $g_2(P)$ at $v_j$ to get another $F$-path and put it back into $\PP_1$. 
\end{itemize}

The number of vertices of $\mathcal{T}$ is reduced by $1$ after each step, while $\mathcal{T}$ still satisfies both properties.
Once the process is complete, we end up with a singleton $\mathcal{T}$ and an $F$-collection $\mathcal{R}$ such that $H_{\mathcal{R}}$ is connected and spans $V$. Note that $\PP_0$ is empty at the end of the process, since exactly one path of $\PP_0$ is used at each step. We merge $\mathcal{R}$ with $\PP_1$ to obtain a new collection $\PP_1'$. Consequently, $H_{\PP_1'}$ is connected.

Let $\PP=\PP_1'\cup \PP_2$. Note that $U_\PP=U_{\PP_1'}\cup U_{\PP_2}=G$, so $\PP$ is an $F$-decomposition of $G$ and $H_\PP$ is connected. The degrees of all vertices of $G$ are even, then so are the degrees of vertices of $H_\PP$, and hence $H_\PP$ is eulerian.
The process of concatenating also ensures that every ray of $\PP$ has length at most $\ell+3$ and that very covered $F$-path of $\PP$ has length at least $\ell$.

It remains to prove that ${\rm conf}(\PP|F) \leq 1/2(\ell + 9)$.
In the following, by saying $\PP_0$ or $\PP_1$, we mean the collection before the process of concatenating.
Recall that $H_{\PP_0}$ is subcubic, so for every vertex $v$, 
the number of $F$-paths with endpoint $v$ in $\PP'_1$ is at most the number $F$-paths with endpoint $v$ in $\PP_1$ plus 3.
Combining with (\ref{eq: H1}) yields 
$d_{\PP'_1}(v) \le d_{\PP_1}(v)+3 \le d_{\PP_2}(v)/4 (\ell+9)$. 
Recall that ${\rm conf}_{\PP_1'}(v|F)\le 1$ and ${\rm conf}_{\PP_2}(v|F)\le 1/4(\ell+9)$.
Hence for every vertex $v$, by definition of conflict ratio we have 
\begin{align*}
{\rm conf}_{\PP}(v|F) & \le \frac{d_{\PP_2}(v){\rm conf}_{\PP_2}(v|F)+d_{\PP_1'}(v){\rm conf}_{\PP_1'}(v|F)}{d_{\PP_2}(v)+d_{\PP_1}(v)}\\
& < {\rm conf}_{\PP_2}(v|F)+ \frac{d_{\PP_1'}(v){\rm conf}_{\PP_1'}(v|F)}{d_{\PP_2}(v)}\\
& \le \frac{1}{4(\ell+9)} 
+ \frac{1}{4(\ell+9)} \\
&\le \frac{1}{2(\ell + 9)}.
\end{align*}
This implies ${\rm conf}(\PP|F)\le 1/2(\ell+9)$, and the lemma follows. 
\end{proof}

The final step is concatenating $F$-paths of $\PP$ to obtain a well-behaved eulerian tour of $G$, which can be done thank to Proposition \ref{thm: Jackson}.

\begin{proof}[Proof of Theorem \ref{theorem: 4euler multi}]
Let $d_\ell=100 \ell L'_{\ell,\varepsilon}$ and $G'=(V,F)$.
We first obtain an $F$-decomposition $\PP$ of $G$ satisfying Lemma \ref{lemma: 4ec multi}. 
For every ray $P_{v|F}$ of $\PP$, each vertex $w\in P_{v|F}$ is a conflict point between $P_{v|F}$ and at most $d_\PP(v)/2(\ell + 9)$ other rays. Hence the number of rays conflicting with $P_{v|F}$ is at most $(\ell + 3) d_\PP(v)/2(\ell + 9) \leq d_\PP(v)/2 - 2$ since $P_{v|F}$ has length at most $\ell+3$.

We wish to apply Proposition \ref{thm: Jackson} to $H_\PP$.
Therefore the task now is to eliminate all loops of $H_\PP$. Let $H_\PP^*$ be the loopless multigraph obtained from $H_\PP$ by subdividing every loop $e=vv$ into $vx_e$ and $x_ev$ by a new vertex $x_e$. 
We associate each $vx_e$ and $x_ev$ with a ray of $P$, where $P\in \mathcal{P}(v)$ is the corresponding $F$-path of $e$.

For every pair of incident vertex-edge $(v,e)$ of $H_\PP^*$, let $S_v(e)$ be the set of all edges of $H_\PP^*$ corresponding to rays conflicting with $P_{v|F}$,  where $P_{v|F}$ is ray of $\PP$ corresponding to $e$. Since two rays of the same $F$-path are non-conflicting, we have $|S_{x_e}(e)|=0$ for every loop $e$ of $H_\PP^*$.
Hence $|S_{v}(e)|\le d_{H_\PP^*}(v)/2 - 2$ if $d_{H_\PP^*}(v)\ge 4$ and $|S_{v}(e)|=0$ if $d_{H_\PP^*}(v)=2$ for every pair of incident vertex-edge $(v,e)$ of $H_\PP^*$.

Let $\mathcal{S}=\{S_v\}_{v\in V}$, then $\mathcal{S}$ is a generalized transition system 
of $H_\PP^*$.
Proposition \ref{thm: Jackson} asserts that $H_\PP^*$ admits an eulerian tour $\mathcal{E}_{H_\PP^*}$ compatible with $\mathcal{S}$, i.e., the corresponding rays of any two consecutive edges of $\mathcal{E}_{H_\PP^*}$ are non-conflicting. 
Clearly, $vx_e$ and $x_ev$ are two consecutive edges of $\mathcal{E}_{H_\PP^*}$ since $x_e$ has degree 2. 
We therefore naturally obtain from $\mathcal{E}_{H_\PP^*}$ an eulerian tour $\mathcal{E}_{H_\PP}$ of $H_\PP$ by replacing $e$ to the segment $vx_ev$ for every loop $e=vv$ of $H_\PP$. 
Hence we naturally obtain from $\mathcal{E}_{H_\PP}$ an eulerian tour ${\mathcal{E}}$ of $G$ by replacing every edge of $\mathcal{E}_{H_\PP}$ by its corresponding $F$-path of $\PP$. Note that every two consecutive (with respect to $\EE$) rays of $\PP$ are non-conflicting.

Let $W$ be a segment of  $\mathcal{E}$ of length at most $\ell$ and consists of only edges of $F$. It remains to prove that $W$ is a path. 
Let $P_1,P_2...,P_r$ be consecutive (with respect to $\EE$) $F$-paths of $\PP$ such that $W$ is a subwalk of  $P_1P_2...P_r$ and $W\cap P_1,W\cap P_r\ne \emptyset$.
If $r\ge 3$ then $W$ must contain entirely $P_2$.
All edges of $W$ belong to $F$, then so does $P_2$.  Hence $P_2$ is a covered $F$-path of length at most $\ell-2$, contrary to that every covered $F$-path of $\PP$ has length at least $\ell$. 
If $r=2$, note that the rays from $v$ of $P_1$ and $P_2$ are non-conflicting, and $W$ is a subwalk of the concatenation of these two rays. Hence $W$ is a path. If $r=1$ then clearly $W$ is a path, the desired conclusion.
\end{proof}


\end{document}